\documentclass[a4paper]{amsart}
\usepackage{graphicx,amsmath,amsfonts,latexsym,amssymb,amsthm,mathrsfs, color}
\usepackage[latin1]{inputenc}
\renewcommand{\Re}{\mathop{\rm Re}}

\newtheorem{thm}{Theorem}[section]

\newtheorem{corollary}[thm]{Corollary}
\newtheorem{definition}[thm]{Definition}
\newtheorem{theorem}[thm]{Theorem}
\newtheorem{lemma}[thm]{Lemma}
\newtheorem{remark}[thm]{Remark}

\begin{document}
\title[]%
{Complex powers for cone differential operators and the heat equation on manifolds with conical singularities}
\author{Nikolaos Roidos}
\address{Institut f\"ur Analysis, Leibniz Universit\"at Hannover, Welfengarten 1, 30167 Hannover, Germany}
\email{roidos@math.uni-hannover.de}

\subjclass[2010]{35K05; 35K65; 35R01; 46B70; 58J35}
\date{\today}

\date{\today}
\begin{abstract}
We obtain left and right continuous embeddings for the domains of the complex powers of sectorial $\mathbb{B}$-elliptic cone differential operators. We apply this result to the heat equation on manifolds with conical singularities and provide asymptotic expansions of the unique solution close to the conical points. We further show that the decomposition of the solution in terms of asymptotics spaces, i.e. finite dimensional spaces that describe the domains of the integer powers of the Laplacian and determined by the local geometry around the singularity, is preserved under the evolution.
\end{abstract}
\maketitle

\section{Introduction}

The domains of the complex powers of sectorial operators play an important role in the regularity theory of partial differential equations (PDEs). Concerning the theory of maximal regularity, under elementary embeddings (see e.g. \cite[(I.2.5.2)]{Am} and \cite[(I.2.9.6)]{Am}), they determine the interpolation spaces of Banach couples appearing in the evolution of linear and quasilinear parabolic problems (see e.g. \cite[Theorem 7.1]{Am1}, \cite[Theorem 2.7]{Ar} or \cite[Theorem 2.1]{CL}). In the case of usual elliptic differential operators, the domains of the complex powers can be determined by standard pseudodifferential theory and in general they are described by fractional Sobolev spaces. Of particular interest is the case of degenerate differential operators, where such knowledge can be applied to the study of PDEs on singular spaces. 

In this article we are interested in conically degenerate operators. This class contains the naturally appearing operators on manifolds with conical singularities. It is well known (see e.g. \cite{GKM}, \cite{Le}, \cite{Sh} or \cite{Schu}) that a cone differential operator which satisfies certain ellipticity assumptions (i.e. being $\mathbb{B}$-elliptic) and is defined as an unbounded operator on an arbitrary (weighted) Mellin-Sobolev space has several closed extensions, called realizations, which differ by a finite dimensional space, called asymptotics space. The structure of the asymptotics space is standard and is determined by the coefficients of the operator close to the singularity. 

Therefore, the description under continuous embeddings of the complex power domains of sectorial $\mathbb{B}$-elliptic cone differential operators is related to the interpolation between Mellin-Sobolev spaces and direct sums of Mellin-Sobolev spaces and asymptotics spaces. We proceed to such an estimate by using only elementary interpolation theory and basic facts concerning the description of the maximal domain of a $\mathbb{B}$-elliptic cone differential operator. Then, embeddings for the domains of the complex powers are recovered by standard theory in the case of sectorial closed extension. For the property of sectoriality, we point out the associated theory and provide an example related to the Laplacian on a conic manifold. 

As a next step, we apply the above result to the theory of degenerate parabolic PDEs. In order to emphasize concrete results, we deal only with linear theory. More precisely, we consider the heat equation on a manifold with conical singularities. However, our complex power result can be applied to non-linear problems on such spaces, such as e.g. the porous medium equation or the Cahn-Hilliard equation, and provide information concerning the asymptotic behavior of the solutions close to the singularity, as we remark later on. 

Concerning now the heat equation on a conic manifold, instead of choosing as usual an appropriate realization of the Laplacian on a Mellin-Sobolev space as e.g. in \cite{Be}, \cite{CSS2}, \cite{RS2} and \cite{Sh}, we pick up a closed extension by going arbitrarily high to the power scale of a specific $R$-sectorial realization of the Laplacian. In this way, by using standard maximal $L^q$-regularity theory, we show well posedness of the problem on spaces consisting of sums of Mellin-Sobolev spaces and asymptotics spaces. 

As a consequence, it is shown that an appropriate decomposition of the initial data in terms of a Mellin-Sobolev space part and an asymptotics space part is preserved under the evolution. Furthermore, the asymptotics space expansion of the solution can become arbitrary long, and hence can provide arbitrarily sharp information concerning the asymptotic behavior of the solution close to the conical tips, depending on the regularity of the right hand side of the heat equation. Hence, e.g. in the case of the homogeneous problem, the complete asymptotic expansion of the solution close to the singularity is recovered, which turns out to be dependent on the local geometry of the cone. 

\subsection*{Acknowledgment} We thank Elmar Schrohe for helpful discussions concerning Theorem \ref{interembed}. 

\section{Basic Maximal $L^{q}$-regularity Theory for Linear Parabolic Problems}

Let $X_{1}\overset{d}{\hookrightarrow} X_{0}$ be a continuously and densely injected complex Banach couple. We start with a basic decay property of the resolvent of an operator which allows us to define complex powers. 

\begin{definition}[Sectorial operator]
Let $\theta\in[0,\pi)$ and denote by $\mathcal{P}(\theta)$ the class of all closed densely defined linear operators $A$ in $X_{0}$ such that 
$$
S_{\theta}=\{\lambda\in\mathbb{C}\,|\, |\arg \lambda|\leq\theta\}\cup\{0\}\subset\rho{(-A)} \quad \mbox{and} \quad (1+|\lambda|)\|(A+\lambda)^{-1}\|_{\mathcal{L}(X_{0})}\leq K_{A,\theta}, \quad \lambda\in S_{\theta},
$$
for some $K_{A,\theta}\geq1$ that is called {\em sectorial bound} of $A$ and depends on $A$ and $\theta$. The elements in $\mathcal{P}(\theta)$ are called {\em (invertible) sectorial operators of angle $\theta$}.
\end{definition}

If $A\in\mathcal{P}(0)$, then $A\in \mathcal{P}(\psi)$ for certain $\psi\in(0,\pi)$, as we can see from \cite[(III.4.6.4)]{Am} and \cite[(III.4.6.5)]{Am}. For any $\rho\geq0$ and $\theta\in(0,\pi)$, let the positively oriented path 
$$
\Gamma_{\rho,\theta}=\{re^{-i\theta}\in\mathbb{C}\,|\,r\geq\rho\}\cup\{\rho e^{i\phi}\in\mathbb{C}\,|\,\theta\leq\phi\leq2\pi-\theta\}\cup\{re^{+i\theta}\in\mathbb{C}\,|\,r\geq\rho\}.
$$
Given any $A\in \mathcal{P}(\theta)$, $\theta\in(0,\pi)$, we can define its complex powers $A^z$ for $\mathrm{Re}(z)<0$ by using the Dunford integral formula, namely
$$
A^{z}=\frac{1}{2\pi i}\int_{\Gamma_{\rho,\theta}}(-\lambda)^{z}(A+\lambda)^{-1}d\lambda \in \mathcal{L}(X_{0}),
$$
for certain $\rho>0$ due to \cite[(III.4.6.4)]{Am}. The above family together with $A^{0} = I$ is a strongly continuous analytic semigroup on $X_{0}$ (see e.g. \cite[Theorem III.4.6.2]{Am} and \cite[Theorem III.4.6.5]{Am}). Moreover, it is easy to verify that each $A^{z}$ is injective for $\mathrm{Re}(z)<0$. Then, we define $A^{w}=(A^{-w})^{-1}$ for $\mathrm{Re}(w)>0$, which are in general unbounded operators (see e.g. \cite[(III.4.6.12)]{Am}). Finally, the imaginary powers of $A$ are defined in a similar way (see e.g. \cite[(III.4.6.21)]{Am}). We refer to \cite[Section III.4.6]{Am} for further details on the complex powers of sectorial operators. 

Recall next the notion of $R$-sectoriality, a boundedness condition for the resolvent of an operator stronger than the standard sectoriality, which plays a fundamental role in the theory of maximal $L^q$-regularity. 
\begin{definition}
Let $\{\epsilon_{k}\}_{k=1}^{\infty}$ be the sequence of the Rademacher functions and $\theta\in[0,\pi)$. An operator $A\in\mathcal{P}(\theta)$ is called {\em $R$-sectorial of angle $\theta$}, if for any choice of $\lambda_{1},...,\lambda_{N}\in S_{\theta}\backslash\{0\}$ and $x_{1},...,x_{N}\in X_0$, $N\in\mathbb{N}\backslash\{0\}$, we have that
\begin{eqnarray*}
\|\sum_{k=1}^{N}\epsilon_{k}\lambda_{k}(A+\lambda_{k})^{-1}x_{k}\|_{L^{2}(0,1;X_0)} \leq R_{A,\theta} \|\sum_{k=1}^{N}\epsilon_{k}x_{k}\|_{L^{2}(0,1;X_0)},
\end{eqnarray*}
for some constant $R_{A,\theta}\geq 1$ that is called {\em $R$-sectorial bound} of $A$ and depends on $A$ and $\theta$.
\end{definition} 

For any $q\in(1,\infty)$ and $\phi\in(0,1)$, denote by $L^{q}(0,T;X_{0})$ the $X_{0}$-valued $L^{q}$-space and by $(\cdot,\cdot)_{\phi,q}$ the real interpolation functor of exponent $\phi$ and parameter $q$. Consider the abstract linear parabolic problem
\begin{eqnarray}\label{AP}
u'(t)+Au(t)&=&f(t), \quad t\in(0,T),\\\label{AP2}
u(0)&=&u_{0}
\end{eqnarray}
where $q\in(1,\infty)$, $T>0$ is finite, $f\in L^{q}(0,T;X_{0})$, $u_{0}\in (X_{1},X_{0})_{\frac{1}{q},q}$ and $-A:X_{1}\rightarrow X_{0}$ is the infinitesimal generator of a bounded analytic semigroup on $X_{0}$. The operator $A$ has {\em maximal $L^{q}$-regularity} if for some $q$ (and hence for all, according to a result by G. Dore) we have that for any $f\in L^{q}(0,T;X_{0})$ and $u_{0}\in (X_{1},X_{0})_{\frac{1}{q},q}$ there exists a unique $u\in W^{1,q}(0,T;X_{0})\cap L^{q}(0,T;X_{1})$ solving \eqref{AP}-\eqref{AP2} that depends continuously on the data $f$, $u_{0}$. Finally, recall the standard embedding of the maximal $L^{q}$-regularity space (see e.g. \cite[Theorem III.4.10.2]{Am}), namely
\begin{gather}\label{embmaxreg}
W^{1,q}(0,T;X_{0})\cap L^{q}(0,T;X_{1})\hookrightarrow C([0,T];(X_{1},X_{0})_{\frac{1}{q},q}).
\end{gather}

If we restrict to Banach spaces belonging to the class of UMD (i.e. having the unconditionality of martingale differences property, see \cite[Section III.4.4]{Am}) then $R$-sectoriality implies maximal $L^q$-regularity (it actually characterizes this property in UMD spaces due to \cite[Theorem 4.2]{W}) as we can see from the following fundamental result.
\begin{theorem}{\rm (Kalton and Weis)}\label{KW}
In a UMD Banach space any $R$-sectorial operator of angle $\theta$ with $\theta>\frac{\pi}{2}$ has maximal $L^{q}$-regularity. 
\end{theorem}
\begin{proof}
This is \cite[Theorem 6.5]{KW1} for the case of $u_{0}=0$. For the general case see e.g. \cite[Theorem 2.7]{Ar}.
\end{proof}

Finally, we note that the property of maximal $L^q$-regularity is preserved on power scales of $R$-sectorial operators in UMD spaces, as we deduce from the following elementary result. 

\begin{lemma}\label{l1}
Let $X_{0}$ be a complex Banach space and let $A:\mathcal{D}(A)\rightarrow X_{0}$ be an $R$-sectorial operator of angle $\theta\in[0,\pi)$. For any $k\in\mathbb{N}\backslash\{0\}$ let 
$$
A_{k}\in \mathcal{L}(\mathcal{D}(A^{k}),\mathcal{D}(A^{k-1})) \quad \text{defined by} \quad A_{k}: x\mapsto Ax, 
$$
where 
$$
\mathcal{D}(A^{k})=(\mathcal{D}(A^{k}),\|A^{k}\cdot\|_{X_{0}})=(\{y\in\mathcal{D}(A^{k-1})\, |\,Ay\in \mathcal{D}(A^{k-1})\},\|A^{k}\cdot\|_{X_{0}}).
$$
Then, $A_{k}$ is also $R$-sectorial of angle $\theta$.
\end{lemma}
\begin{proof}
It is easy to see that $\rho(-A)\subseteq\rho(-A_{k})$ and that $(A+\lambda)^{-1}|_{\mathcal{D}(A^{k-1})}=(A_{k}+\lambda)^{-1}$ for all $\lambda\in\rho(-A)$. Let $\lambda_{1},...,\lambda_{N}\in S_{\theta}\backslash\{0\}$ and $x_{1},...,x_{N}\in X_{0}$, $N\in\mathbb{N}\backslash\{0\}$. By the $R$-sectoriality of $A_{1}$ we have that
\begin{eqnarray*}
\lefteqn{\|\sum_{i=1}^{N}\epsilon_{i}\lambda_{i}(A_{k}+\lambda_{i})^{-1}x_{i}\|_{L^{2}(0,1;\mathcal{D}(A^{k-1}))}}\\
&=& \|\sum_{i=1}^{N}\epsilon_{i}\lambda_{i}(A+\lambda_{i})^{-1}A^{k-1}x_{i}\|_{L^{2}(0,1;X_{0})}\\
&\leq& C\|\sum_{i=1}^{N}\epsilon_{i}A^{k-1}x_{i}\|_{L^{2}(0,1;X_{0})}\\
&=& C\|\sum_{i=1}^{N}\epsilon_{i}x_{i}\|_{L^{2}(0,1;\mathcal{D}(A^{k-1}))}
\end{eqnarray*}
for some constant $C\geq1$.
\end{proof}

\section{Complex Powers for Cone Differential Operators}

Let $\mathcal{B}$ be a smooth $(n+1)$-dimensional, $n\geq1$, manifold with possibly disconnected closed (i.e. compact without boundary) boundary $\partial\mathcal{B}$. Endow $\mathcal{B}$ with a Riemannian metric $g$ which in a collar neighborhood $[0,1)\times\partial\mathcal{B}$ of the boundary admits the form 
\begin{eqnarray}\label{metre2}
g=dx^2+ x^2 h(x,y),
\end{eqnarray}
where $(x,y)\in [0,1)\times\partial \mathcal{B}$ are local coordinates and $x\mapsto h(x,y)$ is a smooth up to $x=0$ family of Riemannian metrics on the cross section $\partial\mathcal{B}$ that does not degenerate up to $x=0$. We call $\mathbb{B}=(\mathcal{B},g)$ {\em conic manifold} or {\em manifold with conical singularities} which are identified with the set $\{0\}\times \partial\mathcal{B}$. When $h$ is independent of $x$ we have straight conical tips, otherwise the tips are warped. Finally, let $\partial\mathcal{B}=\cup_{i=1}^{k_{\mathcal{B}}}\partial\mathcal{B}_{i}$, for certain $k_{\mathcal{B}}\in\mathbb{N}\backslash\{0\}$, where $\partial\mathcal{B}_{i}$ are smooth, closed and connected.

The naturally appearing differential operators on $\mathbb{B}$ degenerate and belong to the class of {\em cone differential operators} or {\em conically degenerate operators}. A cone differential operator $A$ of order $\mu\in\mathbb{N}$ is an $\mu$-th order differential operator with smooth coefficients in the interior $\mathbb{B}^{\circ}$ of $\mathbb{B}$ such that when it is restricted on the collar part $(0,1)\times\partial\mathcal{B}$ it admits the following form 
\begin{gather}\label{A}
A=x^{-\mu}\sum_{k=0}^{\mu}a_{k}(x)(-x\partial_{x})^{k}, \quad \mbox{where} \quad a_{k}\in C^{\infty}([0,1);\mathrm{Diff}^{\mu-k}(\partial\mathcal{B})).
\end{gather}
If $a_{k}$, $k\in\{0,...,\mu\}$, do not depend on $x$ close to zero we say that $A$ has {\em $x$-independent coefficients}.

We associate two special symbols to a cone differential operator. If $(\xi,\eta)$ are the corresponding covariables to the local coordinates $(x,y)\in [0,1)\times \partial\mathcal{B}$ near the boundary, then we define the {\em rescaled symbol} by 
$$
\widetilde{\sigma}_{\psi}^{\mu}(A)(y,\eta,\xi)=\sum_{k=0}^{\mu}\sigma_{\psi}^{\mu-k}(a_{k})(0,y,\eta)(-i\xi)^{k}\in C^{\infty}((T^{\ast}\partial\mathcal{B}\times\mathbb{R})\backslash\{0\}).
$$
Furthermore, the following holomorphic family of differential operators defined on the boundary
$$
\sigma_{M}^{\mu}(A)(\lambda)=\sum_{k=0}^{\mu}a_{k}(0)\lambda^{k} : \mathbb{C} \rightarrow \mathcal{L}(H_{p}^{s}(\partial\mathbb{B}),H_{p}^{s-\mu}(\partial\mathbb{B}))
$$
is called {\em conormal symbol} of $A$, where $\partial\mathbb{B}=(\partial\mathcal{B},h(0))$, $s\in\mathbb{R}$, $p\in(1,\infty)$ and $H_{p}^{s}(\partial\mathbb{B})$ denotes the usual Sobolev space. We may then extend the notion of ellipticity to the case of conically degenerate differential operators as follows. 

\begin{definition}
A cone differential operator $A$ is called {\em $\mathbb{B}$-elliptic} if its standard principal pseudodifferential symbol $\sigma_{\psi}^{\mu}(A)$ is invertible on $T^{\ast}\mathcal{B}^{\circ}\backslash\{0\}$ and $\widetilde{\sigma}_{\psi}^{\mu}(A)$ is also pointwise invertible.
\end{definition}

Cone differential operators act naturally on scales of weighted {\em Mellin-Sobolev spaces} $\mathcal{H}_{p}^{s,\gamma}(\mathbb{B})$, $s,\gamma\in\mathbb{R}$, $p\in(1,\infty)$. Let $\omega\in C^{\infty}(\mathbb{B})$ be a fixed cut-off function near the boundary, i.e. a smooth non-negative function on $\mathcal{B}$ with $\omega=1$ near $\{0\}\times\partial \mathcal{B}$ and $\omega=0$ on $\mathcal{B}\backslash([0,1)\times \partial \mathcal{B})$. Moreover, we assume that in local coordinates $(x,y)\in [0,1)\times \partial\mathcal{B}$, $\omega$ depends only on $x$. Denote by $C_{c}^{\infty}$ the space of smooth compactly supported functions. 
\begin{definition}
For any $\gamma\in\mathbb{R}$ consider the map 
$$
M_{\gamma}: C_{c}^{\infty}(\mathbb{R}_{+}\times\mathbb{R}^{n})\rightarrow C_{c}^{\infty}(\mathbb{R}^{n+1}) \quad \mbox{defined by} \quad u(x,y)\mapsto e^{(\gamma-\frac{n+1}{2})x}u(e^{-x},y). 
$$
Further, take a covering $\kappa_{i}:U_{i}\subseteq\partial\mathcal{B} \rightarrow\mathbb{R}^{n}$, $i\in\{1,...,N\}$, $N\in\mathbb{N}\backslash\{0\}$, of $\partial\mathcal{B}$ by coordinate charts and let $\{\phi_{i}\}_{i\in\{1,...,N\}}$ be a subordinated partition of unity. For any $s\in\mathbb{R}$ and $p\in(1,\infty)$ let $\mathcal{H}^{s,\gamma}_p(\mathbb{B})$ be the space of all distributions $u$ on $\mathbb{B}^{\circ}$ such that 
$$
\|u\|_{\mathcal{H}^{s,\gamma}_p(\mathbb{B})}=\sum_{i=1}^{N}\|M_{\gamma}(1\otimes \kappa_{i})_{\ast}(\omega\phi_{i} u)\|_{H^{s}_p(\mathbb{R}^{n+1})}+\|(1-\omega)u\|_{H^{s}_p(\mathbb{B})}
$$
is defined and finite, where $\ast$ refers to the push-forward of distributions. The space $\mathcal{H}^{s,\gamma}_p(\mathbb{B})$ is independent of the choice of the cut-off function $\omega$, the covering $\{\kappa_{i}\}_{i\in\{1,...,N\}}$ and the partition $\{\phi_{i}\}_{i\in\{1,...,N\}}$.

\end{definition}
Hence, a cone differential operator $A$ of order $\mu$ induces a bounded map
$$
A: \mathcal{H}^{s+\mu,\gamma+\mu}_p(\mathbb{B}) \rightarrow \mathcal{H}^{s,\gamma}_p(\mathbb{B}).
$$
Note that if $s\in \mathbb{N}$, then $\mathcal{H}^{s,\gamma}_p(\mathbb{B})$ is the space of all functions $u$ in $H^s_{p,loc}(\mathbb{B}^\circ)$ such that near the boundary
\begin{eqnarray*}\label{measure}
x^{\frac{n+1}2-\gamma}(x\partial_x)^{k}\partial_y^{\alpha}(\omega(x) u(x,y)) \in L^{p}\big([0,1)\times \partial \mathcal{B}, \sqrt{\mathrm{det}[h(x)]}\frac{dx}xdy\big),\quad k+|\alpha|\le s.
\end{eqnarray*}
 
We recall next some basic facts concerning the domain of a $\mathbb{B}$-elliptic cone differential operator $A$. Further details can be found in \cite{BS}, \cite{GKM}, \cite{Le}, \cite{SS}, \cite{Sh} or \cite{Schu}. We regard $A$ as an unbounded operator in $\mathcal{H}^{s,\gamma}_p(\mathbb{B})$, $s,\gamma\in\mathbb{R}$, $p\in(1,\infty)$, with domain $C_{c}^{\infty}(\mathbb{B}^{\circ})$. The domain of the minimal extension (i.e. the closure) $\underline{A}_{\min}$ of $A$ is given by 
$$
\mathcal{D}(\underline{A}_{\min})=\bigg\{u\in \bigcap_{\varepsilon>0}\mathcal{H}^{s+\mu,\gamma+\mu-\varepsilon}_p(\mathbb{B}) \, |\, x^{-\mu}\sum_{k=0}^{\mu}a_{k}(0)(-x\partial_{x})^{k}(\omega u)\in \mathcal{H}^{s,\gamma}_p(\mathbb{B})\bigg\}.
$$ 
If in addition the conormal symbol of $A$ is invertible on the line $\{\lambda\in\mathbb{C}\,|\, \mathrm{Re}(\lambda)= \frac{n+1}{2}-\gamma-\mu\}$, then we have that
$$
\mathcal{D}(\underline{A}_{\min})=\mathcal{H}^{s+\mu,\gamma+\mu}_p(\mathbb{B}).
$$ 

Concerning the domain of the maximal extension $\underline{A}_{\max}$ of $A$, which as usual is defined by
\begin{gather*}
\mathcal{D}(\underline{A}_{\max})=\bigg\{u\in\mathcal{H}^{s,\gamma}_p(\mathbb{B}) \, |\, Au\in \mathcal{H}^{s,\gamma}_p(\mathbb{B})\bigg\},
\end{gather*}
we have that 
\begin{gather}\label{dmax1}
\mathcal{D}(\underline{A}_{\max})=\mathcal{D}(\underline{A}_{\min})\oplus\mathcal{E}_{A,\gamma}.
\end{gather}
Here $\mathcal{E}_{A,\gamma}$ is a finite-dimensional space independent of $s$, that is called {\em asymptotics space}, which consists of linear combinations of $C^{\infty}(\mathbb{B}^\circ)$ functions that vanish on $\mathcal{B}\backslash([0,1)\times\partial\mathcal{B})$ and in local coordinates $(x,y)$ on the collar part $(0,1)\times\partial\mathcal{B}$ they are of the form $\omega(x)c(y)x^{-\rho}\log^{m}(x)$ where $c\in C^{\infty}(\partial\mathbb{B})$, $\rho\in \mathbb{C}$ and $m\in\mathbb{N}$. 

For the $x^{-1}$ powers $\rho$ describing $\mathcal{E}_{A,\gamma}$ we have that $\rho\in Q_{A,\gamma}$, where $Q_{A,\gamma}$ is a finite set of points in the strip 
\begin{gather}\label{strip}
\bigg\{\lambda\in\mathbb{C}\, |\, \mathrm{Re}(\lambda)\in [\frac{n+1}{2}-\gamma-\mu,\frac{n+1}{2}-\gamma)\bigg\}
\end{gather}
that is determined explicitly by the poles of the recursively defined family of symbols
\begin{gather*}\label{symb1}
g_{0}=f_{0}^{-1}, \quad g_{k}=-(T^{-k}f_{0}^{-1})\sum_{i=0}^{k-1}(T^{-i}f_{k-i})g_{i}, \quad k\in\{1,...,\mu-1\},
\end{gather*}
where 
$$
f_{\nu}(\lambda)=\frac{1}{\nu!}\sum_{i=0}^{\mu}(\partial_{x}^{\nu}a_{i})(0)\lambda^{i}, \quad \nu\in\{0,...,\mu-1\}, \quad \lambda\in\mathbb{C},
$$
and by $T^{\sigma}$, $\sigma\in\mathbb{R}$, we denote the action $(T^{\sigma}f)(\lambda)=f(\lambda+\sigma)$ (see e.g. \cite[(2.7)-(2.8)]{Sh}). The logarithmic powers $m$ are related to the orders of the above poles.

If $A$ has $x$-independent coefficients, then $Q_{A,\gamma}$ coincides with the set of poles of $(\sigma_{M}^{\mu}(A)(\cdot))^{-1}$ in the strip \eqref{strip} and
\begin{gather}\label{dmax}
\mathcal{E}_{A,\gamma}=\bigoplus_{\rho\in Q_{A,\gamma}}\mathcal{E}_{\rho},
\end{gather}
where $\mathcal{E}_{\rho}$, $\rho\in Q_{A,\gamma}$, is a finite-dimensional space independent of $s$ consisting of $C^{\infty}(\mathbb{B}^\circ)$ functions that vanish on $\mathcal{B}\backslash([0,1)\times\partial\mathcal{B})$ and in local coordinates $(x,y)$ on the collar part $(0,1)\times\partial\mathcal{B}$ they are of the form 
\begin{gather}\label{ero}
\omega(x)x^{-\rho}\sum_{i=0}^{m_{\rho}}c_{i}(y)\log^{i}(x)
\end{gather}
with $c_{i}\in C^{\infty}(\partial\mathbb{B})$ and certain $m_{\rho}\in\mathbb{N}$ depending on the order of $\rho$.

Under \eqref{dmax}, let the closed extension $\underline{A}$ of $A$ in $\mathcal{H}^{s,\gamma}_{p}(\mathbb{B})$ be given by
\begin{gather}\label{da}
\mathcal{D}(\underline{A})=\mathcal{D}(\underline{A}_{\min})\oplus\bigoplus_{\rho\in \underline{Q}_{A,\gamma}}\underline{\mathcal{E}}_{\rho},
\end{gather}
where $\underline{Q}_{A,\gamma}\subseteq Q_{A,\gamma}$ is a given subset and $\underline{\mathcal{E}}_{\rho}$ is a subspace of $\mathcal{E}_{\rho}$. The maximal domain structure together with standard properties of interpolation spaces, imply the following result concerning real interpolation between Mellin-Sobolev spaces and direct sums of Mellin-Sobolev spaces and asymptotics spaces, which is inspired by \cite[Lemma 5.2]{RS2}. 

\begin{theorem}\label{interembed}
Let $s,\gamma\in\mathbb{R}$, $\theta\in(0,1)$, $p,q\in(1,\infty)$, $A$ be $\mathbb{B}$-elliptic with $x$-independent coefficients and $\underline{A}$ be the closed extension \eqref{da}. Then, for any $\varepsilon>0$ the following embeddings hold
\begin{eqnarray*}
\lefteqn{\mathcal{H}^{s+\theta\mu+\varepsilon,\gamma+\theta\mu+\varepsilon}_{p}(\mathbb{B})+\bigoplus_{\rho\in \underline{Q}_{A,\gamma}}\underline{\mathcal{E}}_{\rho}}\\
&&\hookrightarrow(\mathcal{H}^{s,\gamma}_{p}(\mathbb{B}),\mathcal{H}^{s+\mu,\gamma+\mu}_{p}(\mathbb{B})\oplus \bigoplus_{\rho\in \underline{Q}_{A,\gamma}}\underline{\mathcal{E}}_{\rho})_{\theta,q}\\
&&\hookrightarrow \mathcal{H}^{s+\theta\mu-\varepsilon,\gamma+\theta\mu-\varepsilon}_{p}(\mathbb{B})+\bigoplus_{\rho\in \underline{Q}_{A,\gamma}}\mathcal{E}_{\rho}.
\end{eqnarray*}
\end{theorem}
\begin{proof}
By \cite[Lemma 3.6]{RS2}, we have that
\begin{gather}\label{poo}
\mathcal{H}^{s+\theta\mu+\varepsilon,\gamma+\theta\mu+\varepsilon}_{p}(\mathbb{B})\hookrightarrow(\mathcal{H}^{s,\gamma}_{p}(\mathbb{B}),\mathcal{H}^{s+\mu,\gamma+\mu}_{p}(\mathbb{B}))_{\theta,q}.
\end{gather}
Standard properties of interpolation spaces (see e.g. \cite[Theorem B.2.3]{Ha}) imply that 
\begin{eqnarray}\nonumber
\lefteqn{(\mathcal{H}^{s,\gamma}_{p}(\mathbb{B}),\mathcal{H}^{s+\mu,\gamma+\mu}_{p}(\mathbb{B}))_{\theta,q}}\\\label{emb1}
&&\hookrightarrow (\mathcal{H}^{s,\gamma}_{p}(\mathbb{B}),\mathcal{H}^{s+\mu,\gamma+\mu}_{p}(\mathbb{B})\oplus\bigoplus_{\rho\in \underline{Q}_{A,\gamma}}\underline{\mathcal{E}}_{\rho})_{\theta,q}.
\end{eqnarray}
Furthermore, 
\begin{gather}\label{emb26}
\bigoplus_{\rho\in \underline{Q}_{A,\gamma}}\underline{\mathcal{E}}_{\rho}\hookrightarrow(\mathcal{H}^{s,\gamma}_{p}(\mathbb{B}),\mathcal{H}^{s+\mu,\gamma+\mu}_{p}(\mathbb{B})\oplus\bigoplus_{\rho\in \underline{Q}_{A,\gamma}}\underline{\mathcal{E}}_{\rho})_{\theta,q}.
\end{gather}
Therefore, the first embedding follows by \eqref{poo}, \eqref{emb1} and \eqref{emb26}.

Concerning the second embedding, since $A$ maps the space
$\bigoplus_{\rho\in \underline{Q}_{A,\gamma}}\underline{\mathcal{E}}_{\rho}$
 into $\mathcal{H}^{s,\gamma}_{p}(\mathbb{B})$, we obtain that
\begin{eqnarray*}
\lefteqn{A: (\mathcal{H}^{s,\gamma}_{p}(\mathbb{B}),\mathcal{H}^{s+\mu,\gamma+\mu}_{p}(\mathbb{B})\oplus \bigoplus_{\rho\in \underline{Q}_{A,\gamma}}\underline{\mathcal{E}}_{\rho})_{\theta,q}}\\
&& \rightarrow (\mathcal{H}^{s-\mu,\gamma-\mu}_{p}(\mathbb{B}),\mathcal{H}^{s,\gamma}_{p}(\mathbb{B}))_{\theta,q}\hookrightarrow \mathcal{H}^{s+(\theta-1)\mu-\varepsilon,\gamma+(\theta-1)\mu-\varepsilon}_{p}(\mathbb{B}),
\end{eqnarray*}
where we have used \cite[Lemma 3.5]{RS2} for the last embedding. Hence, the interpolation space 
$$
(\mathcal{H}^{s,\gamma}_{p}(\mathbb{B}),\mathcal{H}^{s+\mu,\gamma+\mu}_{p}(\mathbb{B})\oplus \bigoplus_{\rho\in \underline{Q}_{A,\gamma}}\underline{\mathcal{E}}_{\rho})_{\theta,q}
$$ 
embeds to the maximal domain of $A$ in $\mathcal{H}^{s+(\theta-1)\mu-\varepsilon,\gamma+(\theta-1)\mu-\varepsilon}_{p}(\mathbb{B})$, i.e.
\begin{eqnarray}\nonumber
\lefteqn{(\mathcal{H}^{s,\gamma}_{p}(\mathbb{B}),\mathcal{H}^{s+\mu,\gamma+\mu}_{p}(\mathbb{B})\oplus \bigoplus_{\rho\in \underline{Q}_{A,\gamma}}\underline{\mathcal{E}}_{\rho})_{\theta,q}}\\\label{eref}
&&\hspace{30mm} \hookrightarrow \mathcal{H}^{s+\theta\mu-\varepsilon,\gamma+\theta\mu-2\varepsilon}_{p}(\mathbb{B})\oplus\bigoplus_{\sigma\in Q_{A,\nu}}\mathcal{E}_{\sigma},
\end{eqnarray}
where $\nu=\gamma+(\theta-1)\mu-\varepsilon$. 

By standard properties of interpolation spaces and according to \eqref{ero}, the operator 
\begin{gather}\label{op}
\omega(x)\bigg(\prod_{\rho\in \underline{Q}_{A,\gamma}}\big(x\partial_{x} + \rho\big)^{m_{\rho}+1}\bigg)
\end{gather}
maps the left hand side of \eqref{eref} to 
$$
(\mathcal{H}^{s-\eta,\gamma}_{p}(\mathbb{B}),\mathcal{H}^{s+\mu-\eta,\gamma+\mu}_{p}(\mathbb{B}))_{\theta,q}\hookrightarrow \mathcal{H}^{s+\theta\mu-\eta-\delta,\gamma+\theta\mu-\delta}_{p}(\mathbb{B}),
$$
for any $\delta>0$ and certain $\eta\in\mathbb{N}$. Therefore, by the construction of \eqref{op}, for the last sum in \eqref{eref} we deduce that $\sigma\in \underline{Q}_{A,\gamma}$.
\end{proof}

\begin{remark}\label{Rmrk}
Theorem \ref{interembed} still holds if we replace $\mathcal{H}^{s+\mu,\gamma+\mu}_{p}(\mathbb{B})$ with $\mathcal{D}(\underline{A}_{\min})$.
\end{remark}

Assume that $A$ has $x$-independent coefficients and let $\underline{A}$ be the closed extension \eqref{da}. Take $k\in\mathbb{N}$, $k\geq1$, and consider the integer powers $\underline{A}^{k}$ of $\underline{A}$ defined as usual by 
$$
\mathcal{D}(\underline{A}^{k})=\big\{u\in \mathcal{D}(\underline{A}^{k-1})\, |\, \underline{A}u\in \mathcal{D}(\underline{A}^{k-1})\big\}.
$$
Since $A^{k}$ is also $\mathbb{B}$-elliptic, by regarding $\underline{A}^{k}$ as a closed extension of $A^{k}$ in $ \mathcal{H}^{s,\gamma}_{p}(\mathbb{B})$ we have that
\begin{gather}\label{dak}
\mathcal{D}(\underline{A}^{k})=\mathcal{D}(\underline{A^{k}}_{\min})\oplus\bigoplus_{\rho\in \underline{Q}_{A^{k},\gamma}}\underline{\mathcal{F}}_{\rho},
\end{gather}
where $\underline{Q}_{A^{k},\gamma}\subseteq Q_{A^{k},\gamma}$ and $\underline{\mathcal{F}}_{\rho}\subseteq\mathcal{F}_{\rho}$ denotes the usual asymptotics space corresponding to the pole $\rho$. Recall that for the minimal domain in general we have that 
\begin{gather*}\label{akmin}
\mathcal{H}^{s+k\mu,\gamma+k\mu}_p(\mathbb{B})\hookrightarrow\mathcal{D}(\underline{A^{k}}_{\min})\hookrightarrow \mathcal{H}^{s+k\mu,\gamma+k\mu-\varepsilon}_p(\mathbb{B})
\end{gather*}
for all $\varepsilon>0$. Then, \cite[(I.2.5.2)]{Am}, \cite[(I.2.9.6)]{Am} and Remark \ref{Rmrk} imply the following.

\begin{corollary}{\rm (Complex powers)}\label{cp}
Let $s,\gamma\in\mathbb{R}$, $p\in(1,\infty)$, $c\geq0$, $k\in\mathbb{N}$, $k\geq1$ and $z\in\mathbb{C}$ with $0<\Re(z)<k$. Assume that $A$ is $\mathbb{B}$-elliptic with $x$-independent coefficients and that for the closed extension $\underline{A}$ given by \eqref{da}, $\underline{A}+c$ is sectorial, i.e. it belongs to the class $\mathcal{P}(0)$. Then, according to \eqref{dak}, for all $\varepsilon>0$ we have that 
\begin{eqnarray*}
\lefteqn{\mathcal{H}^{s+\mu\Re(z)+\varepsilon,\gamma+\mu\Re(z)+\varepsilon}_{p}(\mathbb{B})+\bigoplus_{\rho\in \underline{Q}_{A^{k},\gamma}}\underline{\mathcal{F}}_{\rho}}\\
&&\hookrightarrow\mathcal{D}((\underline{A}+c)^{z})\hookrightarrow
 \mathcal{H}^{s+\mu\Re(z)-\varepsilon,\gamma+\mu\Re(z)-\varepsilon}_{p}(\mathbb{B})+\bigoplus_{\rho\in \underline{Q}_{A^{k},\gamma}}\mathcal{F}_{\rho}.
\end{eqnarray*}
\end{corollary}

As examples of sectorial closed extensions of $\mathbb{B}$-elliptic cone differential operators we refer to \cite[Proposition 1]{CSS2} and \cite[Theorem 4.3]{Sh}. A typical one is obtained by the Laplacian $\Delta$ on $\mathcal{B}$ induced by the metric $g$. This operator in the collar neighborhood $(0,1)\times\partial\mathcal{B}$ near the boundary is of the form 
\begin{eqnarray}\label{Delta}
\Delta = \frac{1}{x^{2}}\big((x\partial_{x})^2 + (n-1+ \frac{x\partial_{x} (\det[h(x)])}{2\det[h(x)]})(x\partial_x )+\Delta_{h(x)}
 \big),
\end{eqnarray}
where $\Delta_{h(x)}$ is the Laplacian on the cross section $\partial\mathcal{B}$ induced by the metric $h(x)$. The conormal symbol of $\Delta$ is given by 
$$
\sigma_{M}(\Delta)(\lambda)=\lambda^{2}-(n-1)\lambda+\Delta_{h(0)}.
$$

Clearly, $(\sigma_{M}(\Delta)(\lambda))^{-1}$ is defined as a meromorphic in $\lambda\in\mathbb{C}$ family of pseudodifferential operators with values in $\mathcal{L}(H_{p}^{s}(\partial\mathbb{B}),H_{p}^{s+2}(\partial\mathbb{B}))$, $s\in\mathbb{R}$, $p\in(1,\infty)$. More precisely, if $\sigma(\Delta_{h(0)})=\{\lambda_{i}\}_{i\in\mathbb{N}}$ is the spectrum of $\Delta_{h(0)}$, then the poles of $(\sigma_{M}(\Delta)(\cdot))^{-1}$ coincide with the set
$$
\bigg\{\frac{n-1}{2}\pm \sqrt{\bigg(\frac{n-1}{2}\bigg)^{2}-\lambda_{i}}\bigg\}_{i\in\mathbb{N}}.
$$
Therefore, the pole zero of $(\sigma_{M}(\Delta)(\cdot))^{-1}$ is always contained in the strip \eqref{strip} provided that $\gamma\in(\frac{n-3}{2},\frac{n+1}{2})$. In this case, denote again by $\mathbb{C}$ the subspace of $\mathcal{E}_{\Delta,\gamma}$ in \eqref{dmax1} under the choice $\rho=m=0$ and $c|_{\partial\mathcal{B}_{i}}=c_{i}$, $c_{i}\in \mathbb{C}$, $i\in\{1,...,k_{\mathcal{B}}\}$, i.e. $\mathbb{C}$ consists of smooth functions that are locally constant close to the boundary. Such a realization can satisfy the property of maximal $L^q$-regularity, as we can see from the following result.

\begin{theorem}\label{RsecD}
Let $s\geq0$, $p\in(1,\infty)$ and the weight $\gamma$ be chosen as 
\begin{eqnarray}\label{choicegg}
\frac{n-3}2<\gamma<\min\bigg\{-1+\sqrt{\bigg(\frac{n-1}{2}\bigg)^{2}-\lambda_{1}} ,\frac{n+1}{2}\bigg\},
\end{eqnarray}
where $\lambda_{1}$ is the greatest non-zero eigenvalue of the boundary Laplacian $\Delta_{h(0)}$. Consider the closed extension $\underline{\Delta}$ of $\Delta$ in $\mathcal{H}_{p}^{s,\gamma}(\mathbb{B})$ with domain 
\begin{gather}\label{DD}
\mathcal{D}(\underline{\Delta})=\mathcal{H}_{p}^{s+2,\gamma+2}(\mathbb{B})\oplus\mathbb{C}.
\end{gather}
Then, for any $\theta\in[0,\pi)$ there exists some $c>0$ such that $c-\underline{\Delta}$ is $R$-sectorial of angle $\theta$.
\end{theorem}
\begin{proof}
This is \cite[Theorem 4.1]{R} or \cite[Theorem 5.6]{RS}.
\end{proof}

\section{The Heat Equation on Manifolds with Conical Singularities}

We consider the following well known linear parabolic equation 
\begin{eqnarray}\label{e1}
u'(t)-\Delta u(t)&=&f(t),\,\,\,\,\,\, t>0,\\\label{e2}
u(0)&=&u_{0},
\end{eqnarray}
for appropriate functions $f$ and $u_{0}$, which describes the heat distribution in a given domain. The above problem, called {\em heat equation}, was treated in \cite{CSS2}, \cite{RS2} and \cite{Sh} on manifolds with straight conical tips and it was shown existence, uniqueness and maximal $L^{q}$-regularity of the solution on Mellin-Sobolev spaces. More precisely, in \cite[Theorem 6]{CSS2} it is shown maximal $L^q$-regularity for \eqref{e1}-\eqref{e2} by employing the minimal extension of the Laplacian on a weighted $L^p$-space. Then, this result is extended to {\em dilation invariant} extensions of the Laplacian in \cite[Theorem 5.8]{Sh}. In \cite{RS2} a non-linear generalization, called {\em porous medium equation}, is considered and it is shown maximal $L^q$-regularity on arbitrary order Mellin-Sobolev spaces \cite[Theorem 4.2]{RS2} as well as on spaces with asymptotics in the sense of the domain of bi-Laplacian \cite[Proposition 7.5]{RS2}. The same problem is treated in \cite{JL} on surfaces with straight conical tips by using the Friedrichs extension of the Laplacian. Finally, we also refer to \cite{Be} for an alternative approach to the problem with similar results, as well as to \cite{Ve} for the properties of the bi-harmonic heat kernel on such spaces.

In order to study the evolution on asymptotics spaces, we consider here the same problem with the difference that the Laplacian is chosen on the power scale defined by the realization \eqref{DD}.

\begin{theorem}\label{t1}
Let $s\geq0$, $k\in\mathbb{N}$, $k\geq1$, $p,q\in(1,\infty)$, $\gamma$ be chosen as in \eqref{choicegg},
$$
f\in L^{q}(0,\infty;\mathcal{D}(\underline{\Delta}^{k-1})) \quad \text{and} \quad u_{0}\in (\mathcal{D}(\underline{\Delta}^{k}),\mathcal{D}(\underline{\Delta}^{k-1}))_{\frac{1}{q},q},
$$ 
where $\underline{\Delta}$ is the realization \eqref{DD}. Then, for each $T>0$ there exists a unique 
$$
u\in W^{1,q}(0,T;\mathcal{D}(\underline{\Delta}^{k-1}))\cap L^{q}(0,T;\mathcal{D}(\underline{\Delta}^{k}))
$$ 
solving the problem \eqref{e1}-\eqref{e2} on $[0,T)\times \mathbb{B}$. Moreover, $u$ depends continuously on $f$ and $u_{0}$.
\end{theorem}
\begin{proof}
Let $\theta\in(\frac{\pi}{2},\pi)$ and $c>0$ such that $c-\underline{\Delta}$ is $R$-sectorial of angle $\theta$ due to Theorem \ref{RsecD}. Consider the following linear degenerate parabolic problem
\begin{eqnarray*}\label{e33}
v'(t)+(c-\underline{\Delta})v(t)&=&e^{-ct}f(t),\,\,\,\,\,\, t\in(0,T),\\\label{e44}
v(0)&=&u_{0}.
\end{eqnarray*}
We regard $\underline{\Delta}$ as an operator from $\mathcal{D}(\underline{\Delta}^{k})$ to $\mathcal{D}(\underline{\Delta}^{k-1})$. Then, the result follows by applying Theorem \ref{KW} and Lemma \ref{l1} to the above problem and then setting $v=e^{-ct}u$. 
\end{proof}

The maximal $L^{q}$-regularity of the solution obtained in the above theorem together with the interpolation results of the previous section can show that the asymprotics space decomposition of the initial data $u_{0}$ in \eqref{e2} can be preserved under the evolution induced by \eqref{e1}. More precisely, by the embedding \eqref{embmaxreg}, the reiteration result \cite[Corollary 7.3]{Haase}, \cite[Lemma 5.2]{RS2}, Remark \ref{Rmrk} and Theorem \ref{t1} we obtain the following.

\begin{corollary}
Assume that the metric $h$ in \eqref{metre2} is independent of $x$. Let $s\geq0$, $k\in\mathbb{N}$, $k\geq1$, $p,q\in(1,\infty)$, $\gamma$ be chosen as in \eqref{choicegg}, $\varepsilon>0$, 
$$
f\in L^{q}(0,\infty;\mathcal{H}^{s+2(k-1),\gamma+2(k-1)}_p(\mathbb{B})) 
$$
and
$$
 \quad u_{0}\in \mathcal{H}^{s+2k-\frac{2}{q}+\varepsilon,\gamma+2k-\frac{2}{q}+\varepsilon}_{p}(\mathbb{B})+\bigoplus_{\rho\in \underline{Q}_{\Delta^{k},\gamma}}\underline{\mathcal{F}}_{\rho},
$$ 
where the asymptotics spaces involving the initial data determine the domain of the $k^{\mathrm{th}}$ power $\underline{\Delta}^{k}$ of the realization \eqref{DD}, i.e. we have that
$$
\mathcal{D}(\underline{\Delta}^{k})=\mathcal{D}(\underline{\Delta^{k}}_{\min})\oplus\bigoplus_{\rho\in \underline{Q}_{\Delta^{k},\gamma}}\underline{\mathcal{F}}_{\rho}
$$
with 
$$
\mathcal{H}^{s+2k,\gamma+2k}_p(\mathbb{B})\hookrightarrow\mathcal{D}(\underline{\Delta^{k}}_{\min})\hookrightarrow \mathcal{H}^{s+2k,\gamma+2k-\delta}_p(\mathbb{B})
$$
for all $\delta>0$, $\underline{Q}_{\Delta^{k},\gamma}\subseteq Q_{\Delta^{k},\gamma}$ and $\underline{\mathcal{F}}_{\rho}\subseteq\mathcal{F}_{\rho}$ according to \eqref{da}. Then, for each $T>0$, for the unique solution of the problem \eqref{e1}-\eqref{e2} on $[0,T)\times \mathbb{B}$ obtained by Theorem \ref{t1} we have that
$$
 u\in C([0,T];\mathcal{H}^{s+2k-\frac{2}{q}-\varepsilon,\gamma+2k-\frac{2}{q}-\varepsilon}_{p}(\mathbb{B})+\bigoplus_{\rho\in \underline{Q}_{\Delta^{k},\gamma}}\mathcal{F}_{\rho}).
$$
\end{corollary}

From the above result we deduce that the more regularity we have for $f$ and $u_{0}$ the better information we obtain concerning the asymptotic behavior of the solution $u$ of \eqref{e1} close to $\{0\}\times\partial\mathcal{B}$. In the case of homogeneous heat equation, i.e. when $f=0$, we can recover the complete asymptotic expansion of the solution in terms of the local geometry around the singularities, as we can see from the following result.

\begin{theorem}
Let $s\geq0$, $p,q\in(1,\infty)$, $\gamma$ be chosen as in \eqref{choicegg}, $f=0$ and $u_{0}\in \mathcal{H}_{p}^{s,\gamma}(\mathbb{B})$. Then, there exists a unique 
$$
u\in C^{\infty}((0,\infty); \mathcal{H}_{p}^{s,\gamma}(\mathbb{B}))\cap C([0,\infty); \mathcal{H}_{p}^{s,\gamma}(\mathbb{B}))\cap C((0,\infty); \mathcal{D}(\underline{\Delta}))
$$ 
solving the problem \eqref{e1}-\eqref{e2} on $\mathbb{B}$, where $\underline{\Delta}$ denotes the realization \eqref{DD}. Furthermore, for any $k\in\mathbb{N}$ we have that
$$
u\in C^{\infty}((0,\infty); \mathcal{D}(\underline{\Delta}^{k})).
$$ 
\end{theorem}
\begin{proof}
Take $\theta\in(\frac{\pi}{2},\pi)$ and $c>0$ such that $c-\underline{\Delta}\in\mathcal{P}(\theta)$ due to Theorem \ref{RsecD}. Consider the following linear degenerate parabolic problem
\begin{eqnarray}\label{e3}
v'(t)+(c-\underline{\Delta})v(t)&=&0,\,\,\,\,\,\, t>0,\\\label{e4}
v(0)&=&u_{0}.
\end{eqnarray}
From \cite[Corollary 3.3.11]{ABHN}, \cite[Theorem 3.7.11]{ABHN} and \cite[Corollary 3.7.21]{ABHN} the above problem admits a unique solution 
$$
v\in C^{\infty}((0,\infty); \mathcal{H}_{p}^{s,\gamma}(\mathbb{B}))\cap C([0,\infty); \mathcal{H}_{p}^{s,\gamma}(\mathbb{B}))\cap C((0,\infty); \mathcal{D}(\underline{\Delta})).
$$ 

Take any $\tau>0$ and consider the problem
\begin{eqnarray}\label{e5}
w'(t)+(c-\underline{\Delta})w(t)&=&0,\,\,\,\,\,\, t>0,\\\label{e6}
w(0)&=&v(\tau).
\end{eqnarray}
By noting that $v(\tau)\in \mathcal{D}(\underline{\Delta})$, we consider $c-\underline{\Delta}$ in \eqref{e5}-\eqref{e6} as an operator in $\mathcal{D}(\underline{\Delta})$ with domain $\mathcal{D}(\underline{\Delta}^{2})$, which due to \cite[Lemma V.1.2.3]{Am} belongs again to $\mathcal{P}(\theta)$. Therefore, from \cite[Corollary 3.7.21]{ABHN} there exists a unique
$$
w\in C^{\infty}((0,\infty); \mathcal{D}(\underline{\Delta}))\cap C([0,\infty); \mathcal{D}(\underline{\Delta}))\cap C((0,\infty); \mathcal{D}(\underline{\Delta}^{2}))
$$ 
solving \eqref{e5}-\eqref{e6}. By uniqueness, we have that $w(t)=v(t+\tau)$ when $t\geq0$. The result now follows by successively applying the above argument and by setting $v=e^{-ct}u$ to \eqref{e3}-\eqref{e4}. 
\end{proof}

\begin{remark}
The porous medium equation on manifolds with straight conical tips was studied in \cite{RS2}. In \cite[Section 7]{RS2} the equation was considered in sums of higher order Mellin-Sobolev spaces and asymptotics spaces and it was shown existence, uniqueness and maximal $L^q$-regularity of the solution (see \cite[Theorem 7.8]{RS2}). Furthermore, the Cahn-Hilliard equation on manifolds with possibly warped conical tips was considered in \cite{RS}, and similar results were shown in terms of higher order Mellin-Sobolev spaces (see \cite[Theorem 4.6]{RS} and \cite[Theorem 5.9]{RS}). By the embedding \eqref{embmaxreg}, Remark \ref{Rmrk} and \cite[Corollary 7.3]{Haase} combined with \cite[Theorem 7.8]{RS2} and \cite[Theorem 4.6]{RS} we can obtain in each case more precise information concerning the asymptotic behavior of the solutions close to the singularities in terms of the description of the domain of the bi-Laplacian. 
\end{remark}

\end{document}